\documentclass[12pt,oneside]{amsart}
     \usepackage{amssymb}
      \usepackage[utf8]{inputenc}
     \usepackage[frenchb,english]{babel}
      \usepackage[T1]{fontenc}
      \usepackage{amsmath}
      \usepackage{amsfonts}
      \usepackage{lmodern}
      \usepackage[left=2cm,right=2cm,top=2cm,bottom=2cm]{geometry}
      \usepackage{variations}
      \usepackage{graphicx} 
      \usepackage{float}	
      \usepackage{geometry} 
      \usepackage{fancyhdr} 
      \usepackage{longtable}
      \usepackage{listings}
      \usepackage{subfigure}
      \usepackage{hyperref} 
      \usepackage[usenames,dvipsnames]{color} 

      \theoremstyle{plain}
      \newtheorem{theorem}{Theorem}[section]

      \newtheorem{Example}[theorem]{Example}
      
      \newtheorem{Proposition}[theorem]{Proposition}
      
      \theoremstyle{definition}

      \theoremstyle{remark}
      \newtheorem{remark}[theorem]{Remark}

      \newcommand{\R}{{\mathbb R}}

      \numberwithin{equation}{section}

      \makeatletter
      \def\@setcopyright{}
      \def\serieslogo@{}
      \makeatother

\begin{document}

   \author{V. Maume-Deschamps}
   \address{Université de Lyon, Université Lyon 1,France, Institut Camille Jordan UMR 5208}
   \email{veronique.maume@univ-lyon1.fr}


   \author{D. Rulli\`ere}

   \address{Université de Lyon, Université Lyon 1,France\\
    Laboratoire SAF EA 2429}

   \email{didier.rulliere@univ-lyon1.fr}
   \author{K. Said}
      \address{Université de Lyon, Université Lyon 2,France, Laboratoires SAF EA 2429 \& COACTIS EA 4161 }
      \email{khalil.said@univ-lyon2.fr}
   

   \title[The Infinite]{Impact of dependence on some multivariate risk indicators}

   \begin{abstract}
    The minimization of some multivariate risk indicators may be used as an allocation method, as proposed in Cénac et al.~\cite{AR2}. The aim of capital allocation is to choose a point in a simplex, according to a given criterion. In a previous paper~\cite{P12015} we proved that the proposed allocation technique satisfies a set of coherence axioms. In the present one, we study the properties and asymptotic behavior of the allocation for some distribution models. We analyze also the impact of the dependence structure on the allocation using some copulas. 
   \end{abstract}

   \date{\today}
   \maketitle
   \section*{Introduction}
   
   Natural phenomena, financial events and risks are usually modeled through random vectors or processes. In these fields, considering dependencies between random variables is necessary. In actuarial science, this issue has led to the development of multivariate risk theories. Indeed, an insurance company is generally exposed to several risks which cannot be assumed to be independent. It is therefore necessary to adopt a multivariate approach that takes into account both the marginal structures of risks and their dependence structure.\\
   
   Multivariate risk theory is based on dependence modeling and includes multivariate ruin probabilities and multivariate risk measures. In the univariate case, ruin probability has been widely studied since the beginning of the 20th century (Lundberg (1903) \cite{lundberg1903}, Cramer (1930) \cite{cramer1930l}). In multivariate contexts, several definitions are proposed for ruin probability. Hult and Lindskog (2006) \cite{Mbt} defined a multivariate ruin probability based on a notion of ruin sets. Cai and Li (2007) \cite{Cai} defined different finite-time ruin probabilities. Risk measure theory was also enriched by some multivariate risk measures definitions. Jouini et al. (2004) \cite{jouini2004} defined vector-valued coherent risk measures. Cousin and Di Bernardino (2013) \cite{Elena2013} introduced some multi-dimensional extensions of usual univariate risk measures such as the Value-at-Risk.
   These measure are multi-dimensional valued, and thus cannot take benefit from a full order. Dhaene et al.~(2012) introduced of real valued family for risk measures for mono-periodic multivariate processes (\cite{AR3}). In multi-periodic context, Cénac et al. (2012) \cite{AR2} defined new multivariate risk indicators as sum of expected local ruin amounts using penalty functions.\\
   
   Capital allocation may be seen as a direct application of multivariate risk theory. Allocating a capital in actuarial contexts means distributing a positive quantity representing an allocation capital in a simplex. For insurance industry the calculation of the regulatory economic capital, which is called the Solvency Capital Requirement, is well controlled and its methodology is almost imposed by the supervisory authorities of the sector. Nevertheless, the allocation of this capital may be considered as an internal exercise for each company, and constitutes a management choice whose success is a key factor for firm performance optimization. It can also be seen as an indicator of its good governance.\\
   
  Several allocation methods were proposed in the actuarial literature. Some are driven from ideas of risk allocation using Euler or Shapley principles (see as examples Tasche (2007) \cite{Euler2} for the Euler allocation principle and Denault(2001) \cite{Denault2001} for the Shapley allocation principle). Others are based on the minimization of some ruin probabilities or multivariate risk indicators.  In this context, Cénac et al. (2012) \cite{AR2} proposed a capital allocation by minimizing some multivariate risk indicators. They presented a numerical optimization algorithm to find the optimal allocation in some general cases. In \cite{AR1} properties and asymptotic behavior of the allocation are studied in some bivariate cases. The impact of dependence on capital allocation was studied for certain cases of Euler's method by Barg{\`e}s et al. (2009) \cite{MarceauTvar}  and Cossette et al. (2012) \cite{MarceauTvar2}.\\
  
  In this paper, we give a further study on the allocation method by minimizing some multivariate risk indicators. The idea of minimizing multivariate risk indicators seems from our point of view more suited for the ORSA\footnote{Own Risk and Solvency
    Assessment} approach of Solvency 2 European norms. In a recent paper \cite{P12015}, we show that this method satisfies a set of coherence axioms. In the present paper, we study the impact of marginal distributions and dependence structure on the optimal allocation. We generalize the results presented in bivariate cases by Cénac et al. (2014) \cite{AR1}. Moreover, we study the impact of the dependence structure on the allocation using some parametric copulas. Our main results are explicit formulas for various classes of multivariate processes, as well as the limit behavior of the allocation as the capital goes to infinity.\\
 
    The paper is organized as follows. In the first section, we recall the allocation method by minimizing multivariate risk indicators. The second section is a generalization in higher dimension of the results presented for the bivariate case in Cénac et al.(2014) \cite{AR1}. We present some explicit formulas obtained for some particular models, and we discuss the asymptotic behavior of the optimal allocation for these models. The impact of the dependence structure on the allocation is studied in Section 3 through the analysis of comonotonic cases and some examples of copulas (FGM, Marshall-Olkins...).
     
   \section{Optimal allocation}
   In this section, we recall one capital allocation principle that consists of minimizing some risk indicators. This method was introduced in Cénac et al. (2012) \cite{AR2}. It is based on the minimization of some multivariate risk indicators.\\
   
     In a multivariate risk framework, we consider an insurance group composed of $d$ branches or business lines. We denote by $u$ the initial capital of the group. Let $X^p$  be a vectorial risk process $X_p=(X^1_p,\ldots,X^d_p)$, where $X^k_p$ corresponds to the losses  of the $k^{th}$ business line during the $p^{th}$ period.  We denote by  $R^k_p$ the reserve of the $k^{th}$ line at time $p$, so: $R^k_p=u_k-\displaystyle\sum_{\ell=1}^{p}X^k_\ell$, where $u_k\in\mathbb{R}^+$ is the initial capital of the $k^{th}$ business line, then $u=u_1+\cdots+u_d$.\\
     
   Cénac et al. (2012) \cite{AR2} defined the two following multivariate risk indicators, given penalty functions $g_k$:
   \begin{itemize}
   \item The indicator $I$: \[ \mathit{I}(u_1,\ldots,u_d)=\sum_{k=1}^{d}{\mathbb{E}\left(\sum_{p=1}^{n}{g_k(R_p^k)1\!\!1_{\{R_p^k<0\}}1\!\!1_{\{\sum_{j=1}^{d}R_p^j>0\}}}\right)} \/,\]
   \item The indicator $J$: 
    \[ \mathit{J}(u_1,\ldots,u_d)=\sum_{k=1}^{d}{\mathbb{E}\left(\sum_{p=1}^{n}{g_k(R_p^k)1\!\!1_{\{R_p^k<0\}}1\!\!1_{\{\sum_{j=1}^{d}R_p^j<0\}}}\right)} \/,\]
   \end{itemize}
   $g_k: \mathbb{R}^-\rightarrow \mathbb{R}^+$ are $C^1$, convex functions with $g_k(0)=0\/, ~g_k(x)\geq 0$ for $x < 0,~k=1,\ldots,d$. They represent the cost that each branch has to pay when it becomes insolvent while the group is solvent for the $I$ indicator, or while the group is also insolvent in the case of the $J$ indicator.\\
   These multivariate risk indicators are a generalization of the mono-periodic indicator family introduced in Dhaene et al.(2012)\cite{AR3}.\\
   
  The indicator $I$ represents the expected sum of penalty amounts of local ruins, knowing that the group remains solvent. In the case of the indicator $J$, the local ruin severities are taken into account only in the case of group insolvency.\\
   
  The idea of the optimal allocation is to allocate some capital $u$ by minimizing these indicators. This is finding an allocation vector $(u_1,\ldots,u_d)$ that minimizes the indicator such as $u=u_1+\cdots+u_d$, where $u$ is the initial capital that need to be shared among all branches.  In our paper \cite{P12015}, we explained why this capital allocation method can be considered as economically coherent.\\  
   
   By using optimization stochastic algorithms, one can estimate the minimum of these risk indicators. Cénac et al. (2012) \cite{AR2} propose a Kiefer-Wolfowitz version of the mirror algorithm as a convergent algorithm under general assumptions to find optimal allocation minimizing the indictor $I$. This algorithm is effective to solve the optimal allocation problem, especially for a large number of business lines, and for allocation over several periods.\\
   
   Since the solvency capital requirement (SCR) is calculated for one year time horizon only, it seems more practical to focus on the case of mono-periodic allocation.\\
   The risk $X_k$ corresponds to the losses of the $k^{th}$ branch during one period. It is a positive random variable in our context. We denote by $u$ the initial capital of the firm, it represents the allocation capital and it can be, as an example, the SCR or another investment capital. The quantity $u_k$ represents the portion of capital allocated to the $k^{th}$ branch, then $\sum_{i=1}^{d}u_i=u$. The simplex $\mathcal{U}^d_u=\{v=(v_1, \ldots, v_d)\in[0,u]^d, \sum_{i=1}^{d}v_i=u\}$ is the set of all possible allocations. For all $i\in \{1,\ldots,d\}$ we denote $\alpha_i=\frac{u_i}{u}$, then, $\sum_{i=1}^{d}\alpha_i=1$. For $(u_1,\ldots,u_d)\in\mathcal{U}^d_u$, we define also the reserve of the $k^{th}$ business line at the end of the period as $R^k=u_k-X_k$.  The aggregate sum of risks is $ S=\sum_{i=1}^{d}{X_i} $, and let  $ S^{-i}=\sum_{j=1;j\neq i}^{d}{X_j} $ for all $i\in \{1,\ldots,d\}$. Finally, $F_Z$ is the distribution function of a random variable $Z$, $\bar{F}_Z$ is its survival function and $f_Z$ its density function.\\
   
   In order to ensure the existence and uniqueness of the indicators minimum in $\mathcal{U}^d_u$, we suppose that for at least one $k\in\{1,\ldots,d\}$, $g_k$ is strictly convex. We also assume that the joint density $f_{(X_k,S)}$ support contains $[0,u]^2$ (see \cite{AR2}). In this case, the indicators $I$ and $J$ are strictly convex and admit a unique minimum.\\  

   Now we focus on the optimality condition for the indicators $I$ and $J$.\\
   For an initial capital $u$, and an optimal allocation minimizing the multivariate risk indicator $I$, we look for $u^*\in \mathbb{R}_+^d$ such that:
   \[\mathit{I}(u^*)=\underset{v_1+\cdots+v_d=u}{\inf}\mathit{I}(v), ~~v\in \mathbb{R}_+^d  \/.\]
    We assume that the functions $g_k$ are differentiable and such that for all $k\in\{1,\ldots,d\}$, the derivative $g^\prime_k(u_k-X_k)$ admits a moment of order one, and that the random vector $(X_k,S)$ has a joint density distribution denoted by $f_{(X_k,S)}$. Under these assumptions, the risk indicators $I$ and $J$ are differentiable, and in this case, we can calculate the following gradients: 
         \begin{align*}
         (\nabla I(v))_i&=\sum_{k=1}^{d}\int_{v_k}^{+\infty}g_k(v_k-x)f_{X_k,S}(x,u)dx
         + \mathbb{E}[g^\prime_i(v_i-X_i)1\!\!1_{\{X_i>v_i\}}1\!\!1_{\{S\leq u\}}]\\
         \mbox{and,}~~ (\nabla J(v))_i&=\sum_{k=1}^{d}\int_{v_k}^{+\infty}g_k(v_k-x)f_{X_k,S}(x,u)dx
               + \mathbb{E}[g^\prime_i(v_i-X_i)1\!\!1_{\{X_i>v_i\}}1\!\!1_{\{S\geq u\}}]\/.
         \end{align*} 
   Using the Lagrange multipliers method, we obtain an optimality condition verified by the unique solution to this optimization problem:
   \begin{equation}
   \mathbb{E}[g^\prime_i(u_i-X_i)1\!\!1_{\{X_i>u_i\}}1\!\!1_{\{S\leq u\}}]=\mathbb{E}[g^\prime_i(u_j-X_j)1\!\!1_{\{X_j>u_j\}}1\!\!1_{\{S\leq u\}}],~~\forall j\in\{1,\ldots,d\}^2
   \label{Optcg}\/.
   \end{equation}
   A natural choice for penalty functions is the ruin severity: $g_k(x)=|x|$. In that case, and if the joint density $f_{(X_k,S)}$ support contains $[0,u]^2$, for at least one $k\in\{1,\ldots,d\}$, our optimization problem  has a unique solution.\\
   
   In Maume-Deschamps et al. (2015) \cite{P12015}, we showed that in the case of penalty functions $g_k(x)=|x|~ \forall k\in \{1,\ldots,d\}$, and for continuous random vector $(X_1,\ldots,X_d)$ such that the joint density $f_{(X_k,S)}$ support contains $[0,u]^2$, for at least one $k\in\{1,\ldots,d\}$, the optimal allocation by minimization of the indicators $I$ and $J$ is a symmetric riskless full allocation. It satisfies also the properties of comonotonic additivity, positive homogeneity, translation invariance, monotonicity, and continuity.\\
    
  We may write the indicators as follows:
 \begin{align*}
 \mathit{I}(u_1,\ldots,u_d)&=\sum_{k=1}^{d}{\mathbb{E}\left({\arrowvert R^k \arrowvert1\!\!1_{\{R^k<0\}}1\!\!1_{\{\sum_{i=1}^{d}R^i\geq0\}}}\right)}\\
 &=\sum_{k=1}^{d}{\mathbb{E}\left({(X_k-u_k)1\!\!1_{\{X_k>u_k\}}1\!\!1_{\{\sum_{i=1}^{d}X_i\leq u\}}}\right)}=\sum_{k=1}^{d}{\mathbb{E}\left({(X_k-u_k)^+1\!\!1_{\{S\leq u\}}}\right)}\/,
 \end{align*}
 and,
 \begin{align*}
 \mathit{J}(u_1,\ldots,u_d)&=\sum_{k=1}^{d}{\mathbb{E}\left({\arrowvert R^k \arrowvert1\!\!1_{\{R^k<0\}}1\!\!1_{\{\sum_{i=1}^{d}R^i\leq0\}}}\right)}\\
 &=\sum_{k=1}^{d}{\mathbb{E}\left({(X_k-u_k)1\!\!1_{\{X_k>u_k\}}1\!\!1_{\{\sum_{i=1}^{d}X_i\geq u\}}}\right)}=\sum_{k=1}^{d}{\mathbb{E}\left({(X_k-u_k)^+1\!\!1_{\{S\geq u\}}}\right)}\/.
 \end{align*} 
In the special case where $g_k(x)=|x|$, \ref{Optcg} becomes:
\begin{equation}\label{ZoneO}
\mathbb{P}\left( X_i>u_{i}, S\leq u\right) = \mathbb{P}\left( X_j>u_{j}, S\leq u\right) , \forall (i,j)\in\{{{1,2,\ldots,d}}\}^{2}\/.
\end{equation}
For the $J$ indicator, this condition can be written:
\begin{equation}\label{ZoneV}
\mathbb{P}\left( X_i>u_{i}, S\geq u\right) = \mathbb{P}\left( X_j>u_{j}, S\geq u\right) , \forall (i,j)\in\{{{1,2,\ldots,d}}\}^{2}\/.
\end{equation}
Some explicit and semi-explicit formulas for the optimal allocation can be obtained with this optimality condition. Our problem reduces to the study of this allocation depending on the nature of the distributions of the risk $X_k$ and on the form of dependence between them.
   \section{Some general results in the independence case}
   In this section we generalize the results presented in dimension 2 by Cénac et al. (2014) in the first section of their paper \cite{AR1} to higher dimension.\\
   The results presented here give explicit forms to the optimal allocation for some specific distributions. This could be used as a benchmark to test optimization algorithms convergence.\\ 
   
    We also get some asymptotic results, when the capital $u$ goes to infinity. We study both the exponential and the sub-exponential cases, and we determine the difference between their asymptotic behavior for exponential and Pareto distributions cases.\\
   
   We consider from now on that the penalty functions are identical and equal to the severity of local ruin $g_k(x)=g(x)=|x|, \forall k\in\{1,2,\ldots,d\}$. The optimality conditions for minimizing the multivariate risk indicators $I$ and $J$ are given respectively by Equations (\ref{ZoneO}) and (\ref{ZoneV}). In this section we focus on the independence case.\\ 
   
   Recall that $\alpha_i=\frac{u_i}{u}\in[0,1]$, so that when $u\rightarrow +\infty$, we may consider convergent subsequences in the proofs below. By abuse of notation, we consider $\underset{u\rightarrow+\infty}{\lim}\alpha_i$. In fact, we consider a convergent subsequence and get the existence of the limit by obtaining the uniqueness of the limit point.
   \subsection{Independent exponentials}~~\\
 ~~Assume $X_1,X_2,\ldots,X_d$ are independent exponential random variables with respective parameters $0<\beta_1<\beta_2<\cdots<\beta_d$. Remark that in the
 particular case where $\beta_1=\beta_2=\cdots=\beta_d$, the optimal allocation is $u_1=\cdots= u_d= u/d$.
 \begin{Proposition}[The optimal allocation for the indicator $I$]
 \label{prp21}
 The allocation minimizing the risk indicator $I$ is the unique solution in $\mathcal{U}^d_u$, of the following equations system:
  \begin{equation}
  h(\beta_i\alpha_i)-h(\beta_j\alpha_j)
   -\sum_{\ell=1}^{d}{A_\ell h(\beta_\ell)[h(\alpha_i\cdot(\beta_i-\beta_\ell) )-h(\alpha_j\cdot(\beta_j-\beta_\ell) )]}=0,\forall (i,j)\in\{{{1,2,\ldots,d}}\}^{2}
   \label{EIZO}\/,
 \end{equation} 
 where h is the function defined by $ h(x)= \exp(-u\cdot x)$, and $A_\ell$ denotes the constants \text{$ A_\ell=\displaystyle\prod_{j=1,j\neq \ell}^{d}{\dfrac{\beta_j}{\beta_j-\beta_\ell}}$}, for 
  $\ell=1,\ldots,d$.      
 \end{Proposition}
 \begin{proof}
 If $X_i\sim\mathcal{E}(\beta_i)$ are independent exponential random variables, then $S^{-i} =\displaystyle\sum_{j=1;j\neq i}^{d}{X_j}$ have a generalized Erlang distribution with parameters $(\beta_1,\beta_2,\ldots,\beta_{i-1},\beta_{i+1},\ldots,\beta_d)$, so we write:
 \begin{align*}
  \mathbb{P}\left( X_i>u_{i},\sum_{j=1}^{d}{X_j}\leq u\right) &= \mathbb{P}\left( X_i>u_{i}\right) -\mathbb{P}\left( X_i>u_{i},\sum_{j=1}^{d}{X_j}>u\right)\\ 
  & = \bar{F}_{X_i}(u_{i})-\bar{F}_{X_i}(u)-\int_{u_{i}}^{u}\bar{F}_{S^{-i}}(u-s)f_{X_i}(s) \mathrm{ds}\\
  &=h(\beta_i\alpha_i)-h(\beta_i)-\sum_{\ell=1}^{d}{A_\ell h(\beta_\ell)h(\alpha_i\cdot(\beta_i-\beta_\ell) )}+\sum_{\ell=1}^{d}{A_\ell h(\beta_i)} \\
  &=h(\beta_i\alpha_i)-\sum_{l=1}^{d}{A_\ell h(\beta_\ell)h(\alpha_i\cdot(\beta_i-\beta_\ell) )} \/,
  \end{align*} 
  because, $\bar{F}_{X_i}(x)=e^{-\beta_ix}$,  $\bar{F}_{S^{-i}}(x)=\displaystyle\sum_{\ell=1,\ell\neq i}^{d}{(\displaystyle\prod_{j=1,j\neq \ell,j\neq i}^{d}{\dfrac{\beta_j}{\beta_j-\beta_\ell}})e^{-\beta_\ell x}}$ and $\displaystyle\sum_{l=1}^{d}{A_\ell}=1$.\\ 
  The survival function of the generalized Erlang distribution with parameters $(\beta_1,\beta_2,\ldots,\beta_d)$ is given by:
  \[ \bar{F}_{X}(x)=\sum_{\ell=1}^{d}{(\displaystyle\prod_{j=1,j\neq \ell}^{d}{\dfrac{\beta_j}{\beta_j-\beta_\ell}})e^{-\beta_\ell x}}=\sum_{\ell=1}^{d}{A_\ell e^{-\beta_\ell x}}\/. \]
The optimal allocation is the unique solution in $\mathcal{U}^d_u$, of the following equations system:
 \[
  \mathbb{P}\left( X_i>u_{i},\displaystyle\sum_{k=1}^{d}{X_k}\leq u\right) = \mathbb{P}\left( X_j>u_{j},\displaystyle\sum_{j=k}^{d}{X_k}\leq u\right) , \forall (i,j)\in\{{{1,2,\ldots,d}}\}^{2} 
 \/,\]
which leads to (\ref{EIZO}).
 \end{proof}
 The resulting system is a system of nonlinear equations, which can be solved numerically.
 \begin{Proposition}[The asymptotic optimal allocation for the indicator $I$]
 \label{pro3.2} When the capital $u$ goes to infinity, the asymptotic optimal allocation satisfies:\[\underset{u\to\infty}\lim{\left(\frac{u_1}{u},\frac{u_2}{u},\ldots,\frac{u_d}{u}\right)}=
     \left(\dfrac{\dfrac{1}{\beta_i}}{\displaystyle\sum_{j=1}^{d}{\dfrac{1}{\beta_j}}}\right)_{i=1,2,\ldots,d}
     \/.\]
 \end{Proposition}
 \begin{proof}
 Equations system (\ref{EIZO}) is equivalent to:
 \begin{equation}
 \forall (i,j)\in\{{{1,2,\ldots,d}}\}^{2},~~h(\beta_i\alpha_i)[1-\displaystyle\sum_{\ell=1}^{d}{A_\ell h((1-\alpha_i)\beta_\ell)}] =
 h(\beta_j\alpha_j)[1-\displaystyle\sum_{\ell=1}^{d}{A_\ell h((1-\alpha_j)\beta_\ell)}]\/.
 \label{eq2}
 \end{equation}
 Firstly, remark that for all $i \in\{{1,2,\ldots,d}\}$, $\underset{u\to\infty}\lim{\sup\frac{u_i}{u}}<1 $ because if this result was not satisfied then there would exist $i\in\{1,\ldots,d\}$ such that  $\underset{u\to\infty}\lim{\frac{u_i}{u}}=\underset{u\to\infty}\lim{\alpha_i}=1 $, taking if necessary a convergent subsequence of $\alpha_i$.  For all $ j\neq i$,  $\underset{u\to\infty}\lim{\frac{u_j}{u}}=\underset{u\to\infty}\lim{\alpha_j}=0 $, and Equations system (\ref{ZoneO}) cannot be satisfied in this case.\\
 Equations system (\ref{eq2}) is equivalent to:
 \[\forall (i,j)\in\{{{1,2,\ldots,d}}\}^{2},~~h(\beta_i\alpha_i-\beta_j\alpha_j)=\frac{1-\sum_{\ell=1}^{d}{A_\ell h((1-\alpha_j)\beta_\ell)}}{1-\sum_{\ell=1}^{d}{A_\ell h((1-\alpha_i)\beta_\ell)}}\/,
 \]
 the right side of the last equations system tends to 1 when $u$ tends to $\infty$, therefore, we deduce that
 $\underset{u\to\infty}\lim{h(\beta_i\alpha_i-\beta_j\alpha_j)}=1$ and consequently:
 \[ \forall (i,j)\in\{{{1,2,\ldots,d}}\}^{2},\underset{u\to\infty}\lim{\alpha_i}=\frac{\beta_j}{\beta_i}\underset{u\to\infty}\lim{\alpha_j}\/, \]
 then, for all $i \in\{{1,2,\ldots,d}\}$:
  \[ \underset{u\to\infty}\lim{\alpha_i}=\dfrac{\dfrac{1}{\beta_i}}{\sum_{j=1}^{d}{\dfrac{1}{\beta_j}}}\/. \]
 \end{proof}
 \begin{remark}
 Based on the above result, we can conclude that asymptotically:
  \begin{itemize}
   \item since $\beta_i<\beta_j$ we have $\alpha_i>\alpha_j$, this means that we allocate more capital to the most risky business line.
   \item $\alpha_i$ is a decreasing function of $\beta_i$. This observation is consistent with the previous conclusion.
   \item $\alpha_j$ is an increasing function of $\beta_i$ for $j\neq i$.
  \end{itemize} 
 \end{remark}
 \begin{Proposition}[The optimal allocation for the indicator $J$] The allocation minimizing the risk indicator $J$ is the unique solution in $\mathcal{U}^d_u$, of the following equations system: 
 \begin{equation}
  \forall (i,j)\in\{{{1,2,\ldots,d}}\}^{2},~~~~
     \sum_{\ell=1}^{d}{A_\ell h(\beta_\ell)[h(\alpha_i\cdot(\beta_i-\beta_\ell) )-h(\alpha_j\cdot(\beta_j-\beta_\ell) )]}=0
     \/.\label{EIZV}
  \end{equation}
  \end{Proposition}
  \begin{proof}
  The proof is similar to that of Proposition \ref{prp21}.
  \end{proof}
  \begin{Proposition}[The asymptotic optimal allocation for the indicator $J$] When the capital $u$ goes to infinity, the optimal allocation minimizing the risk indicator $J$ is the following:
   \[\underset{u\to\infty}\lim{\frac{u_1}{u}}=1
         ~and~
        \underset{u\to\infty}\lim{\frac{u_{j}}{u}}=0
              ~\forall j \in \{2,3,\ldots,d\}\/.\] 
    \end{Proposition}
   \begin{proof}
    Equations system (\ref{EIZV}) is equivalent to the following one:     \[ \forall (i,j)\in\{{{1,2,\ldots,d}}\}^{2},~~~~
         \sum_{\ell=1}^{d}{A_\ell\cdot h((1-\alpha_i)\cdot(\beta_\ell-\beta_i) )}=\sum_{\ell=1}^{d}{A_\ell\cdot h(\alpha_j\cdot(\beta_j-\beta_\ell)+\beta_\ell-\beta_i )}\/.
      \]
      If $\varlimsup\frac{u_1}{u}<1$, as $u$ goes to $+\infty$ in the equations of the previous system for $i=1$ we get, taking if necessary a convergent subsequence,
      \begin{equation}
      ~\forall j \in \{2,3,\ldots,d\}, ~~
      \underset{u\to\infty}\lim{\sum_{\ell=1}^{d}{A_\ell\cdot h(\alpha_j\cdot(\beta_j-\beta_\ell)+\beta_\ell-\beta_1 )}}=A_1\/.
      \label{veronique}
      \end{equation}
     The first terms of these equations can be decomposed into three parts as follows:
     \begin{align*}
     \underset{u\to\infty}\lim{\sum_{\ell=1}^{d}{A_\ell\cdot h(\alpha_j\cdot(\beta_j-\beta_\ell)+\beta_\ell-\beta_1 )}}&=A_1\cdot\underset{u\to\infty}\lim{ h(\alpha_j\cdot(\beta_j-\beta_1))}\\&+\sum_{\ell=2}^{j}{A_\ell}\cdot
             \underset{u\to\infty}\lim{ h(\alpha_j\cdot (\beta_j-\beta_\ell)+\beta_\ell-\beta_1 )}\\&+\sum_{\ell=j+1}^{d}{A_\ell}\cdot\underset{u\to\infty}\lim{ h(\alpha_j\cdot(\beta_j-\beta_\ell)+\beta_\ell-\beta_1 )}\/.
     \end{align*}
       For all,  $j>1$, $\displaystyle\sum_{\ell=2}^{j}{A_\ell}\cdot\underset{u\to\infty}\lim{ h(\alpha_j\cdot(\beta_j-\beta_\ell)+\beta_\ell-\beta_1 )}=0$, because $\beta_\ell-\beta_1>0$ and \text{$\beta_j-\beta_\ell\geq 0$} for $ \ell\in \{2,3,\ldots,j\}$. Moreover, $\displaystyle\sum_{\ell=j+1}^{d}{A_\ell}\cdot\underset{u\to\infty}\lim{ h(\alpha_j\cdot(\beta_j-\beta_\ell)+\beta_\ell-\beta_1 )}=0$, because for all \text{$\ell\in \{j+1,j+2,\ldots,d\},$} $\alpha_j\cdot(\beta_j-\beta_\ell)+\beta_\ell-\beta_1 = (\beta_\ell-\beta_j)(1-\alpha_j)+\beta_j-\beta_1>0$. So that \ref{veronique} leads to $\forall j\in\{1,\ldots,d\}$ \text{$\underset{u\to\infty}\lim{h(\alpha_j\cdot(\beta_j-\beta_1))}=1$}.
       We deduce that $~\forall j \in \{2,3,\ldots,d\}; ~~\alpha_j\overset{u\to\infty}{=}o(\frac{1}{u})$. This contradicts the necessary condition: $\underset{u\to\infty}{\lim}{\sum_{\ell=1}^{d}{\alpha_j}}=1$.   
   \end{proof}
 \subsection{Some distributions of the sub-exponential family}~~\\
 In most cases of risk distributions, we cannot give explicit or semi-explicit optimal allocation formulas, the difficulty comes from the lack of a simple form of the risks sum $S$ and its joint distribution with each risk $X_i$. In this section, we present asymptotic results (as $u$ goes to infinity) for the optimal allocation in the case of some distributions of the sub-exponential family. In this way, we generalize results of \cite{AR1} to higher dimension.\\
 We recall the sub-exponential distributions family definition, consisting in distributions of positive support, with a distribution function that satisfies: 
 \[ \frac{\overline{F^{*2}}(x)}{\overline{F}(x)}\overset{x\to+\infty}{\longrightarrow}2\/,\]
 where $\overline{F^{*2}}$ is the convolution of $\overline{F}$.\\
 In Asmussen (2000) \cite{RuinP}, it is proven that the sub-exponential distributions satisfy also the following relation, for all $d\in\mathbb{N}^*$:
 \[ \frac{\overline{F^{*d}}(x)}{\overline{F}(x)}\overset{x\to+\infty}{\longrightarrow}d\/,\]
 where $\overline{F^{*d}}$ is the $d^{th}$convolution of $\overline{F}$.\\
We shall use the following theorem proved in Cénac and al. 
 \begin{theorem}[Sub-exponential distributions \cite{AR1}] \label{thZO}
 Let $X$ be a random variable with sub-exponential distribution $F_X$, $Y$ a random variable with support $\mathbb{R}^+$, independent of $X$, and $(u,v)\in (\mathbb{R}^+)^2$, such that:
 \begin{itemize}
 \item there exists $0<\kappa_1<\kappa_2<1$ such that for $u$ large enough, $\kappa_1\leq\frac{v}{u}\leq\kappa_2$, 
 \item 
 $ \frac{\bar{F}_X(y)}{\bar{F}_X(x)}\overset{x\to+\infty}{=}O(1)\/, \mbox{ if } y=\Theta(x)$\footnote{For $(x,y)\in\mathbb{R}^{2+}$, we shall denote $y=\Theta(x)$ if there exist $0<C_1\leq C_2<\infty$, such that for $x$ large enough, $C_1\leq\frac{y}{x}\leq C_2$ }.
 \end{itemize}
 Then, \[ \underset{u\to\infty}\lim{\frac{\mathbb{P}(X\geq v,~X+Y\geq u)}{\bar{F}_X(u)}}=1 \/.\]
 \end{theorem}  
 \subsubsection{The asymptotic behavior}~~\\
 Here we examine the asymptotic behavior of the optimal allocation by minimizing the indicators $I$ and $J$ in the cases of some sub-exponential distributions.\\
  In what follows, $(u_1,\ldots,u_d)$ denotes the optimal allocation of $u$ associated to the risk indicator $I$.
  \begin{theorem}
   Let $(X_1,X_2,\ldots,X_d)$ be continuous positive and independent random variables.\\ Assume that $\forall (i,j)\in\{{{1,2,\ldots,d}}\}^{2}$:
   \begin{enumerate}
   \item $\bar{F}_{X_i}(x)\stackrel{ x\rightarrow+\infty}{=}\Theta(\bar{F}_{X_j}(x))$,
   \item $\bar{F}_{X_i}(s)\stackrel{s\rightarrow +\infty}{=}o(\bar{F}_{X_i}(t))$, if $t=o(s)$,
   \end{enumerate}
   then, there exist $\kappa_1>0$ and $\kappa_2<1$ such that, 
   \begin{align}
      \kappa_1\leq \frac{u_\ell}{u} \leq \kappa_2 ~~~ \forall \ell\in \{1,2,\ldots,d\}\/.
      \label{k1k22}
   \end{align}
   \label{th2.6}
  \end{theorem}
  Note that the first condition of Theorem \ref{th2.6} is not satisfied for exponential distributions. However, Pareto distributions satisfy the hypothesis of Theorem \ref{th2.6}.
  \begin{proof}
   Taking if necessary a convergent subsequence, we assume that $\exists i\in\{1,\ldots,d\}$ such that: $\frac{u_i}{u}\stackrel{u\rightarrow +\infty}{\longrightarrow}1$ or $\frac{u_i}{u}\stackrel{u\rightarrow +\infty}{\longrightarrow}0$, the first case implies that foll all $j\neq i$, $\frac{u_j}{u}\stackrel{u\rightarrow +\infty}{\longrightarrow}0$, then, it is sufficient to prove that the existence of an $i\in\{1,\ldots,d\}$ such that $\frac{u_i}{u}\stackrel{u\rightarrow +\infty}{\longrightarrow}0$ is impossible.\\Let us assume the existence of $i\in{\{1,\ldots,d\}}$, such that: $\frac{u_i}{u}\stackrel{u\rightarrow +\infty}{\longrightarrow}0$.\\
    Then, $\exists j\in{\{1,\ldots,d\}\backslash i}$ such that $\underset{u\rightarrow +\infty}{\lim}\frac{u_j}{u}\in]0,1]$, therefore, $u_j\stackrel{u\rightarrow +\infty}{\longrightarrow}+\infty$ and $u_i\stackrel{u\rightarrow +\infty}{=}o(u_j)$.\\
    Using Assumptions (1) and (2), we deduce that:
    \begin{align}
    \frac{\bar{F}_{X_j}(u_j)}{\bar{F}_{X_i}(u_i)}\stackrel{u\rightarrow +\infty}{\longrightarrow}0
    \label{eq21}\/.
    \end{align}
    The optimality condition \eqref{ZoneO} can also be written for all $j\neq i$ as follows:
    \begin{align}
    \bar{F}_{X_i}(u_{i})-\mathbb{P}\left( X_i>u_{i},S>u\right)=\bar{F}_{X_j}(u_{j})-\mathbb{P}\left( X_i>u_{j},S>u\right) 
    \label{opti2}\/.
    \end{align}
    That presents a trivial contradiction if $u_i$ remains bounded.\\
   Now, assume that $u_i\rightarrow+\infty$. Recall that $S^{-i}=\displaystyle\sum_{k=1,k\neq i}^{d}{X_k}$, then: \[ 
   \mathbb{P}\left( X_i>u_{i},S>u\right)=\mathbb{P}\left( X_i>u_{i},S^{-i}>\sqrt{u},S>u\right)+\mathbb{P}\left( X_i>u_{i},S^{-i}<\sqrt{u},S>u\right)
    \/.\]
    We have:
   \[ \mathbb{P}\left( X_i>u_{i},S^{-i}>\sqrt{u},S>u\right)\leq \mathbb{P}\left( X_i>u_{i}\right)\mathbb{P}\left(S^{-i}>\sqrt{u}\right)\stackrel{u\rightarrow +\infty}{=}o(\bar{F}_{X_i}(u_{i}))\/. \] 
   Using assumption (2) and since $u_i=o(u)$, 
   \[ \mathbb{P}\left( X_i>u_{i},S^{-i}<\sqrt{u},S>u\right)\leq\bar{F}_{X_i}(u-\sqrt{u})\stackrel{u\rightarrow +\infty}{=}o(\bar{F}_{X_i}(u_{i})). \]
   We deduce that:\begin{align}
   \mathbb{P}\left( X_i>u_{i},S>u\right)  \stackrel{u\rightarrow +\infty}{=}o(\bar{F}_{X_i}(u_{i}))
   \label{eq22}\/.
   \end{align}
   We remark also that: 
   \begin{align}
      \mathbb{P}\left( X_j>u_{j},S>u\right)  \stackrel{u\rightarrow +\infty}{=}O(\bar{F}_{X_j}(u_{j}))\stackrel{u\rightarrow +\infty}{=}o(\bar{F}_{X_i}(u_{i}))
      \label{eq23}\/.
      \end{align}
  Equation (\ref{opti2}) leads to: 
  \[ 1-\underbrace{\frac{\mathbb{P}\left( X_i>u_{i},S>u\right)}{\bar{F}_{X_i}(u_{i})}}_{T_1}=\underbrace{\frac{\bar{F}_{X_j}(u_{j})}{\bar{F}_{X_i}(u_{i})}}_{T_2}-\underbrace{\frac{\mathbb{P}\left( X_j>u_{j},S>u\right)}{\bar{F}_{X_i}(u_{i})}}_{T_3} \/.\]
  Now, relations: (\ref{eq22}), (\ref{eq21}), and (\ref{eq23}), imply that $T_1$, $T_2$, and $T_3$, go to zero, and this is a contradiction.         
  \end{proof}
  \begin{Proposition}
  \label{subth1}
Let $X_1\/,\ldots\/, X_d$ be continuous, positive and independent random variables such that the support of the density of $(X_i\/,S)$ is $(\R^+)^2$. Let $(u_1\/,\ldots\/, u_d)$ be the optimal allocation of $u$ associated to the risk indicator $I$. We assume:
\begin{enumerate}
\item there exist $0<\kappa_1<\kappa_2<1$ such that for all $i=1\/,\ldots\/, d$ and for all $u\in\R^+$,
$$\kappa_1\leq \frac{u_i}u\leq \kappa_2\/,$$
\item for all $i = 1\/, \ldots\/, d$, if $y=y(x)$ is such that 
$$0< \liminf_{x\rightarrow\infty}\frac{y}x\leq \limsup_{x\rightarrow\infty}\frac{y}x<1$$
then
$$\frac{\overline{F}_{X_i}(x)}{\overline{F}_{X_i}(y)} \stackrel{x\rightarrow\infty}{\longrightarrow} 0 \/.$$
\end{enumerate}
Then, for all $i\/,j=1\/,\ldots\/,d$,
$$\frac{\overline{F}_{X_i}(u_i)}{\overline{F}_{X_j}(u_j)}\stackrel{u\rightarrow \infty}{\longrightarrow} 1\/.$$
 \end{Proposition}
 Assumptions of Proposition \ref{subth1} are satisfied for distributions of exponential type (see remark~(\ref{rem3.9}) below). Its application gives another proof to Proposition \ref{pro3.2}. In contrast, Proposition \ref{subth1} cannot be used for Pareto distributions.
\begin{proof}
Following the lines of the proof of Theorem \ref{th2.6}, take $0<\gamma<1-\kappa_2$,
$$\mathbb{P}(X_i > u_i\/, \ S>u) = \mathbb{P}(X_i >u_i \/, \ S^{-i} > \gamma u \/, \ S>u) + \mathbb{P}(X_i >u_i \/, \ S^{-i} > \gamma u \/, \ S>u) \/.$$
As before, 
$$\mathbb{P}(X_i >u_i \/, \ S^{-i} > \gamma u \/, \ S>u)\leq \mathbb{P}(X_i > u_i ) \mathbb{P}(S^{-i} > \gamma u) \stackrel{u\rightarrow +\infty}{=} o(\overline{F}_{X_i}(u_i)) \/.$$
On the other hand,
\begin{eqnarray*}
\mathbb{P}(X_i >u_i \/, \ S^{-i} > \gamma u \/, \ S>u) & \leq & \mathbb{P}(X_i> u-\gamma u)\\
	&=& \overline{F}_{X_i}((1-\gamma)u) \stackrel{u\rightarrow +\infty}{=} o(\overline{F}_{X_i}(u_i)) \\
	&& \mbox{because} \ 0<\frac{\kappa_1}{1-\gamma}\leq \frac{u_i}{(1-\gamma)u}\leq \frac{\kappa_2}{1-\gamma}<1 \/.\\
\end{eqnarray*}
So that, $\mathbb{P}(X_i > u_i\/, \ S>u) \stackrel{u\rightarrow +\infty}{=} o(\overline{F}_{X_i}(u_i))$ and the same computation gives $\mathbb{P}(X_j > u_j\/, \ S>u) \stackrel{u\rightarrow +\infty}{=} o(\overline{F}_{X_j}(u_j))$. Now, $u_i$ and $u_j$ satisfy Equation (1) and thus,
$$1 + o(1) \stackrel{u\rightarrow +\infty}{=} \frac{\overline{F}_{X_j}(u_j)}{\overline{F}_{X_i}(u_i)} + o(1) \frac{\overline{F}_{X_j}(u_j)}{\overline{F}_{X_i}(u_i)}\/.$$
This implies that $\frac{\overline{F}_{X_j}(u_j)}{\overline{F}_{X_i}(u_i)}$ is bounded from above and thus
$$ \frac{\overline{F}_{X_j}(u_j)}{\overline{F}_{X_i}(u_i)}  \stackrel{u\rightarrow +\infty}{=} 1 + o(1) \/.$$
\end{proof}
  \begin{remark}\label{rem3.9}
  We remark that the hypothesis of Proposition \ref{subth1} are satisfied for distribution of {\em exponential type}, that is distributions verifying:\[ \bar{F}_{X_i}(x)=\Theta(e^{-\mu_ix}) \/.\]
  Indeed, in this case, we have $0<\varliminf\frac{u_i}{u}\leq\varlimsup\frac{u_i}{u}<1,~~\forall i\in\{1,\ldots,d\}$. In fact, if $u_i\stackrel{u\rightarrow +\infty}{=}o(u)$, then, $\exists j\in\{1,\ldots,d\}\backslash\{i\}$ such that $\frac{u_j}{u}\rightarrow \kappa\in]0,1]$, so $u_i\stackrel{u\rightarrow +\infty}{=}o(u_j)$.\\
  Since $\mu_i,\mu_j>0$, $\mu_iu_i\stackrel{u\rightarrow +\infty}{=}o(\mu_ju_j)$. 
  As in the proof of Theorem \ref{th2.6}, we get:
  \begin{align*}
     \mathbb{P}\left( X_i>u_{i},S>u\right)  \stackrel{u\rightarrow +\infty}{=}o(\bar{F}_{X_i}(u_{i}))
     \end{align*}
  and, 
  \begin{align*}
       \mathbb{P}\left( X_j>u_{j},S>u\right)  \stackrel{u\rightarrow +\infty}{=}o(\bar{F}_{X_j}(u_{j}))\/.
       \end{align*}
  From the optimality condition, \[ \frac{\bar{F}_{X_i}(u_{i})}{e^{-\mu_iu_i}}-\underbrace{\frac{\mathbb{P}\left( X_i>u_{i},S>u\right)}{e^{-\mu_iu_i}}}_{T_1}=\underbrace{\frac{\bar{F}_{X_j}(u_{j})}{e^{-\mu_iu_i}}}_{T_2}-\underbrace{\frac{\mathbb{P}\left( X_j>u_{j},S>u\right)}{e^{-\mu_iu_i}}}_{T_3}\/, \]
  which is absurd because as $u\rightarrow+\infty$:\begin{itemize}
  \item $\bar{F}_{X_i}(u_{i})=\Theta(e^{-\mu_iu_i})$,
  \item $T_1=o(1)$,
  \item $T_2=\frac{\bar{F}_{X_j}(u_{j})}{e^{-\mu_ju_j}}e^{-\mu_ju_j+\mu_iu_i}\rightarrow 0$, since $\mu_iu_i=o(\mu_ju_j)$,
  \item and $T_3=o(1)e^{-\mu_ju_j+\mu_iu_i}\rightarrow 0$.
  \end{itemize}
  \end{remark}
 \begin{Proposition}[The asymptotic optimal allocation for the indicator $I$]
 \label{subth2}
  Under the same conditions of Theorem \ref{th2.6}, and if, for all $i\in\{1,\ldots,d\}$, $F_{X_i}$ is a sub-exponential distribution, that verifies:
  \[ \frac{\bar{F}_{X_i}(y)}{\bar{F}_{X_i}(x)}\stackrel{x\rightarrow+\infty}{=}O(1),~~\text{for}~~0<\kappa_1\leq\frac{y}{x}\leq\kappa_2<1 \/.\]
   Then, by minimizing the $I$ indicator, $u_i$ and $u_j$ satisfy:
  \begin{equation}
        \bar{F}_{X_i}(u_i)-\bar{F}_{X_i}(u)\stackrel{u\rightarrow +\infty}{=}\bar{F}_{X_j}(u_j)-\bar{F}_{X_j}(u)+o(\bar{F}_{X_i}(u))\/.
        \label{eqsub2}
        \end{equation}
   \end{Proposition}
   Proposition \ref{subth2} is applicable in the case of Pareto distributions. So, we will use it for determining the optimal asymptotic allocation, for independent risks of Pareto distributions in the next subsection.
   \begin{proof}
   The proof of this theorem is a direct application of Theorems \ref{thZO} and \ref{th2.6}.
    \end{proof}
Now, we focus on the asymptotic optimal allocation by minimizing the risk indicator $J$, and we study the case of sub-exponential distribution family. 
\begin{Proposition}[The asymptotic optimal allocation for the indicator $J$] \label{prpZV}
Let $(X_1,X_2,\ldots,X_d)$ be continuous positive and independent random variables, such that there exists $i\in\{1,\ldots,d\}$ with a sub-exponential distribution, the optimal capital allocation by minimizing the $J$ indicator $(u_1,\ldots,u_d)$ verifies, for all $~j\neq i$:
\[ \underset{u\to\infty}\lim{\frac{\mathbb{P}(X_j\geq u_j,~S\geq u)}{\bar{F}_{X_i}(u)}}=1 \/.\]
\end{Proposition}
 \begin{proof}
 The solution to (\ref{ZoneV}) satisfies :
    \[ 
    \forall j\in\{{{1,2,\ldots,d}}\}, ~~~
    \frac{\mathbb{P}\left( X_i>u_{i},S\geq u\right)}{\mathbb{P}\left( X_i>u\right)} =
    \frac{\mathbb{P}\left( X_j>u_{j},S\geq u\right))}{\mathbb{P}\left( X_i>u\right)} 
     \/.\]
       When $u$ goes to $+\infty$, and using Theorem \ref{thZO}, we obtain Proposition \ref{prpZV}.
 \end{proof}
 \subsubsection{Application to Pareto independent distributions}~~\\
 We consider $d$ independent random variables 
  $(X_1,X_2,\ldots,X_d)$ of Pareto distributions, with parameters   $(a,b_i)_{\{i=1,2,\ldots,d\}}$ respectively, such that $b_1>b_2>\cdots>b_d>0$. Therefore, these distributions will be characterized by densities and survival functions of the following forms:
  \[ f_{X_i}(x)=\frac{a}{b_i}\left(1+\frac{x}{b_i}\right)^{-a-1} \/,\]
  and \[ \bar{F}_{X_i}(x)=\left(1+\frac{x}{b_i}\right)^{-a} \/.\]
   \begin{Proposition}[The asymptotic optimal allocation minimizing $I$]\label{prop3.12}Asymptotically, the unique solution to (\ref{ZoneO}) satisfies:
     \[ \forall (i,j)\in\{{{1,2,\ldots,d}}\}^2,~~~~\left(\frac{\underset{u\to\infty}\lim{\alpha_i}}{b_i}\right)^{-a}-\left(\frac{\underset{u\to\infty}\lim{\alpha_j}}{b_j}\right)^{-a}=\left(\frac{1}{b_i}\right)^{-a}-\left(\frac{1}{b_j}\right)^{-a} \/.\]
      \label{GPDO}
  \end{Proposition}
  \begin{proof}
  Follows from Proposition \ref{subth2}.
    \end{proof}
   \begin{Proposition}[The asymptotic optimal allocation minimizing $J$]\label{prop3.13} 
      The unique solution to~(\ref{ZoneV}) satisfies: 
                \[ \underset{u\to\infty}\lim{\alpha_1}=1
          ~~and~~ \underset{u\to\infty}\lim{\alpha_i}=0,\forall i\in\{{{2,3,\ldots,d}}\}
          \/. \]
    \end{Proposition}
    \begin{proof}
    We suppose that $\exists \in\{1,\ldots,d\}$ such that $0<\underset{u\to\infty}\lim{\alpha_j}<1$.\\
    From Theorem \ref{thZO} : \[ \underset{u\to\infty}\lim{\frac{\mathbb{P}(X_j\geq u_j,S\geq u)}{\bar{F}_{X_j}(u)}}=1 \/.\]
    On the other hand, and applying Proposition \ref{prpZV} in the Pareto distributions case, we get for $i\in\{1,\ldots,d\}\setminus\{j\}$:
     \[ \frac{\mathbb{P}(X_j\geq u_j,S\geq u)}{\bar{F}_{X_j}(u)}= \frac{\mathbb{P}(X_j\geq u_j,S\geq u)}{\bar{F}_{X_i}(u)}\cdot\frac{\bar{F}_{X_i}(u)}{\bar{F}_{X_j}(u)}\stackrel{ u\rightarrow+\infty}{\sim}\frac{\bar{F}_{X_i}(u)}{\bar{F}_{X_j}(u)}=\left(\frac{1+\frac{u}{b_i}}{1+\frac{u}{b_j}}\right)^{-a} \/,\]
  then, for $i\in\{1,\ldots,d\}\setminus\{j\}$: \[ \underset{u\to\infty}\lim{\frac{\mathbb{P}(X_j\geq u_j,S\geq u)}{\bar{F}_{X_j}(u)}}= \left(\frac{b_j}{b_i}\right)^{-a}\neq 1\/. \]   
 That is absurd. We deduce that
  $\forall i\in\{{{1,2,\ldots,d}}\}$:\[\underset{u\to\infty}\lim{\alpha_i}\in\{0,1\}\/, \]
  and since $\sum_{i=1}^{d}{\alpha_i}=1$, then, there is a unique $i$ such that $~ \underset{u\to\infty}\lim{\alpha_i}=1$ and for all $~j\neq i$ $~ \underset{u\to\infty}\lim{\alpha_j}=0$.\\ 
  Let us recall the definition of the order stochastic dominance, as it is presented in Shaked and Shanthikumar (2007)\cite{OrdreSto}. For random variables $X$ and $Y$, $Y$ first-order stochastically dominates $X$ if and only if:
  \[\bar{F}_X(x)\leq\bar{F}_Y(x), ~~\forall x\in\mathbb{R}^+ \/,\]
  and in this case we denote: $X\preccurlyeq_{st}Y$.\\ 
  Now, clearly  $X_d\preccurlyeq_{st}\cdots\preccurlyeq_{st}X_2\preccurlyeq_{st}X_1$ because $b_1>b_2>\cdots>b_d>0$.\\
  In \cite{P12015} (Proposition 3.10) we have proved that the optimal capital allocation satisfies the monotonicity property. We deduce that $\alpha_1\geq\cdots\geq\alpha_d$, and thus $\underset{u\to\infty}\lim{\alpha_1}=1$.
    \end{proof}
    \subsection{Comparison between the asymptotic behaviors in exponential and sub-exponential cases}
 Let us consider the following risk indicator:
    \begin{align*}
     \mathit{I_{loc}}(u_1,\ldots,u_d)=\sum_{k=1}^{d}{\mathbb{E}\left({(X_k-u_k)1\!\!1_{\{X_k>u_k\}}}\right)}=\sum_{k=1}^{d}{\mathbb{E}\left({(X_k-u_k)^+}\right)}\/.
     \end{align*}
With this indicator, the impact of each branch on the group is not taken into account. This indicator thus takes only \textit{local} effects into account. If, for any $k=1\/,\ldots \/,d$, the random vector $(X_k\/,S)$ admits a density whose support contains $[0\/,u]^2$, then the optimality condition associated to $ \mathit{I_{loc}}$ gives:
\[ \mathbb{P}(X_i>u_i)=\mathbb{P}(X_j>u_j) \/.\]
We remark that this corresponds to the asymptotic result for indicator $I$ of Proposition \ref {subth1}, where $ \frac{\overline{F}_{X_j}(u_j)}{\overline{F}_{X_i}(u_i)}  \stackrel{u\rightarrow +\infty}{=} 1 + o(1) $. In other words, in the case of {\em exponential type} and independent risks, the group effect is asymptotically negligible. This behavior is also clear in Proposition \ref{pro3.2}, where we found the asymptotic capital allocation for $I$ for independent exponential distributions as:
$$u_i=\dfrac{\dfrac{1}{\beta_i}}{\sum_{j=1}^{d}{\dfrac{1}{\beta_j}}}u,~\text{for all}~i=1\/,\ldots\/,d\/,$$ as for $I_{loc}$.\\

Proposition \ref{subth2} shows that for independent sub-exponential distributions the asymptotic behavior is different. Indeed, in Equation (\ref{eqsub2}), the terms $\bar{F}_{X_i}(u)$ and $\bar{F}_{X_j}(u)$  lead to take into account the group effect. In the Pareto distributions case as example, recall that Proposition \ref{GPDO} gives that the asymptotic behavior of the optimal allocation for $I$ is described by :
$$\left(\frac{\underset{u\to\infty}\lim{\alpha_i}}{b_i}\right)^{-a}-\left(\frac{\underset{u\to\infty}\lim{\alpha_j}}{b_j}\right)^{-a}=\left(\frac{1}{b_i}\right)^{-a}-\left(\frac{1}{b_j}\right)^{-a}\/,$$
whereas the optimal allocation for $\mathit{I_{loc}}$ for independent Pareto distributions is given by $\alpha_i=\frac{b_i }{\sum_{\ell=1}^{d}b_\ell}$, for all $i\in\{1,\ldots,d\}$.\\
    
For optimal allocation by minimizing the $J$ indicator, the asymptotic behavior is identical for the two families of distributions, the riskiest branch is considered as first responsible of the overall ruin, and thus, the optimal solution is to allocate the entire capital $u$ to this business line.
    
\section{The impact of the dependence structure}
In this section, we focus on the impact of the dependence structure on the optimal allocation. We study at first the impact of mixture exponential-gamma to construct a correlated Pareto  distributions, compared to the independence case presented in the previous section. Then we analyze the optimal allocation in the case of comonotonic risks. The last sub-section is devoted to the study of the impact of the dependence nature on the optimal allocation, using some bivariate models with copulas.
\subsection{Correlated Pareto}~~\\
Let $(X_1,\ldots,X_d)$ be a mixture of exponential distributions such that for all $i\in \{1,2,\ldots,d\}$, 
 $ X_i\sim\mathcal{E}(\beta_i\theta)$, with ($\beta_1<\beta_2\cdots<\beta_d$), and
$\theta\sim \Gamma(a,b)$. Therefore, $X_i$ have survival functions of the form:
\[ \bar{F}_{X_i}(x)=\int_{0}^{\infty}\bar{F}_{X_i|\Theta=\theta}f_\Theta(\theta)d\theta=\int_{0}^{\infty}e^{-\beta_i\theta x}f_\Theta(\theta)d\theta =\left(1+\frac{\beta_ix}{b}\right)^{-a}\/, \]
consequently, $X_i$ have Pareto distribution of parameters $\left(a,\frac{b}{\beta_i}\right)$. They are conditionally independents. So, the idea is conditioning on the random variable $\theta $ and then integrate the formulas found for the case of independent exponential distributions. This model has been studied in e.g.\cite{oakes1989bivariate, Yeh2007}. Recall $ A_\ell=\displaystyle\prod_{j=1,j\neq \ell}^{d}{\dfrac{\beta_j}{\beta_j-\beta_\ell}}$, $\ell=1,\ldots,d$.   
  \begin{Proposition}[The optimal allocation for the indicator $I$]
  The optimal allocation minimizing the multivariate risk indicator $I$ is the unique solution in $\mathcal{U}^d_u$, of the following equations system:
   \[ 
      \forall (i,j)\in\{{{1,2,\ldots,d}}\}^{2}\/, \]
      \begin{equation}
      s(\beta_i\alpha_i)-s(\beta_j\alpha_j)
        -\sum_{\ell=1}^{d}{A_\ell [s(\alpha_i\beta_i+(1-\alpha_i)\beta_\ell) )-s(\alpha_j\beta_j+(1-\alpha_j)\beta_\ell) )]}=0
        \/,\label{eq10}
      \end{equation}
      where $s$ is the function defined by $s(x)=(1+x\frac{u}{b})^{-a}$ and $\alpha_i=\frac{u_i}{u}$ for all $i\in\{1,\ldots,d\}$. 
     \end{Proposition}
   \begin{proof}
   It suffices to integrate Equations System (\ref{EIZO}), multiplied by the density function of $\theta$.
  \end{proof}
 \begin{Proposition}[The asymptotic optimal allocation for the indicator $I$]\label{prop4.2}  When the capital $u$ goes to infinity, the optimal allocation by minimization of the risk indicator $I$ is the unique solution in $\mathcal{U}^d_u$ of the following equations system:
 \[ \forall (i,j)\in\{{{1,2,\ldots,d}}\}^{2}\/, \]
 \begin{equation}
                                 (\beta_i\alpha_i)^{-a}-(\beta_j\alpha_j)^{-a}-\sum_{\ell=1}^{d}{A_\ell [(\alpha_i\beta_i+(1-\alpha_i)\beta_\ell)^{-a}- (\alpha_j\beta_j+(1-\alpha_j)\beta_\ell)^{-a} ]}=0
         \/.\label{eq11}
       \end{equation}
 \end{Proposition}
  \begin{proof}
  We divide Equations System (\ref{eq10}) by s(1), and let $u$ go to $+\infty$ to get Equations System~(\ref{eq11}).
  \end{proof}
  Proposition \ref{prop4.2} shows the impact of the dependence related to the mixture. Indeed, in the case of independent Pareto distributions, of parameters $\left(a,\frac{b}{\beta_i}\right)_{i=1,\ldots,d}$, the asymptotic optimal allocation for the indicator $I$ is given by Proposition \ref{prop3.12} as the solution of the equations system:\[ \forall (i,j)\in\{{{1,2,\ldots,d}}\}^2,~~~~\left(\beta_i\alpha_i\right)^{-a}-\left(\beta_j\alpha_j\right)^{-a}=\left(\beta_i\right)^{-a}-\left(\beta_j\right)^{-a} \/.\]
  Each equation in this system depends only on two risks, unlike the mixture case, where the equations of Equations System (\ref{eq11}), depend on all the risks.
  
  \begin{Proposition}[The optimal allocation for the indicator $J$] 
  The optimal allocation minimizing the multivariate risk indicator $J$ is the unique solution in $\mathcal{U}^d_u$, of the following equations system: 
  \begin{equation}
       \forall (i,j)\in\{{{1,2,\ldots,d}}\}^{2},
          \sum_{\ell=1}^{d}{A_\ell [s(\alpha_i\beta_i+(1-\alpha_i)\beta_\ell) )-s(\alpha_j\beta_j+(1-\alpha_j)\beta_\ell) )]}=0
            \/.\label{ZvDPGCorr}
          \end{equation}
   \end{Proposition}
   \begin{proof}
    It suffices to integrate Equations System (\ref{EIZV}), multiplied by the density function of $\theta$.
  \end{proof}
   \begin{Proposition}[The asymptotic optimal allocation for the indicator $J$] \label{prop4.4}When the capital $u$ goes to infinity, the optimal allocation by minimization of the risk indicator $J$, is the unique solution in $\mathcal{U}^d_u$ of the following equations system:
    \begin{equation}
          \forall (i,j)\in\{{{1,2,\ldots,d}}\}^{2},~~
                       \sum_{\ell=1}^{d}{A_\ell [(\alpha_i\beta_i+(1-\alpha_i)\beta_\ell)^{-a}- (\alpha_j\beta_j+(1-\alpha_j)\beta_\ell)^{-a} ]}=0
                       \label{asy10}\/.
             \end{equation}
   \end{Proposition}
   \begin{proof}
    We divide Equations system (\ref{ZvDPGCorr}) by s(1), and we let $u$ fo to $+\infty$ to get Equations system~(\ref{asy10}).
   \end{proof}  
   Proposition \ref{prop4.4} shows that for the indicator $J$, the asymptotic behaviour of the optimal capital allocation takes into account the mixture effect. In fact, for independent Pareto distributions, we have proved in Proposition \ref{prop3.13}, that we allocate the entire capital $u$ to the riskiest branch $ X_1 $, while the asymptotic optimal allocation in the correlated Pareto distributions case is the solution of Equations system (\ref{asy10}).
\subsection{Comonotonic risks}~~\\
The concept of comonotonic random variables is related to the studies of Hoeffding (1940) \cite{Hoeff1940} and Fréchet (1951) \cite{Frechet1951}. Here we use the definition of comonotonic risks as it was first mentioned in the actuarial literature in Borch (1962) \cite{Borch1962}.\\
A vector of random variables $(X_1,X_2,\ldots,X_n)$ is comonotonic if and only if there exists a random variable $Y$ and non-decreasing functions $\varphi_1,\ldots,\varphi_n$ such that: 
\[ (X_1,\ldots,X_n)\overset{d}{=}(\varphi_1(Y),\ldots,\varphi_n(Y)) \/.\]
In the case where the risks $X_1$,\ldots,$X_d$ are comonotonic, we may give explicit formulas for the optimal allocation minimizing the multivariate risk indicators $I$ and $J$, and for some risk models. For that, we use the existence of a uniform random variable $U$ such that $X_i=F^{-1}_{X_i}(U)$ for all $i\in\{1,\ldots,d\}$, and $S=\sum_{i=1}^{d}F^{-1}_{X_i}(U)=\varphi(U)$, where $\varphi(t) = \sum_{i=1}^{d}F^{-1}_{X_i}(t)$, $\varphi$ is a non-decreasing function.\\
The main result of this section is given below. 
\begin{Proposition}\label{prop:comonotone}
Let $X_1$,\ldots,$X_d$ be comonotonic  risks, with increasing distribution functions and support containing $[0\/,u]$. The optimal allocations for indicators $I$, $J$ and $I_{loc}$ coincide, they are given by $(u_1\/,\ldots \/, u_d)\in\mathcal{U}_u^d$ and 
$$F_{X_i}(u_i) = F_{X_j}(u_j) \ \forall \ i\/, j = 1\/, \ldots \/, d \/.$$ 
\end{Proposition}
\begin{proof}
Let us denote: $w_i=F_{X_i}(u_i)$, $v=\varphi^{-1}(u)$, $M_i=\max(w_i\/,v)$. The indicators $I$ and $J$ may be rewritten for $(u_1\/,\ldots \/, u_d)\in\mathcal{U}_u^d$:
\begin{eqnarray*}
I(u_1\/,\ldots\/, u_d) &=& \sum_{i=1}^d \mathbb{E}\left((F_{X_i}^{-1}(U) -u_i){\bf 1}_{\{U\geq w_i\/, \ U\leq v\}}\right)\\
J(u_1\/,\ldots\/,u_d) &=&\sum_{i=1}^d \mathbb{E}\left((F_{X_i}^{-1}(U) -u_i){\bf 1}_{\{U\geq M_i\}} \right)\/.
\end{eqnarray*}
We remark that since $(u_1\/,\ldots \/, u_d)\in\mathcal{U}_u^d$, and $F_{X_i}$ is strictly increasing for all $i\in\{1,\ldots,d\}$, we cannot have that $w_i<v$ for all $i$, so that, $I$ is not trivially equal to $0$. We use Lagrange multiplier to get that the minimum of $I$ and $J$ are reached in $\mathcal{U}_u^d$ respectively for:
\begin{itemize}
\item $\mathbb{P}(U\geq w_i\/, \ U\leq v) = \mathbb{P}(U\geq w_j\/, \ U\leq v)$, for $i\/,j =1\/,\ldots\/, d$,
\item $\mathbb{P}(U\geq M_i) = \mathbb{P}(U\geq M_j)$, for $i\/,j =1\/,\ldots\/, d$.
\end{itemize} 
These equality are acheaved if and only if $w_i=w_j=v$ or in other words if $F_{X_i}(u_i)=F_{X_j}(u_j)$. We remark that the minimum of $I$ is then $0$. 
\end{proof}
The following three corollaries are direct applications of Proposition \ref{prop:comonotone} to some particular cases.
\begin{Example}[Comonotonic exponential model]
For comonotonic risks of exponential distributions $X_i\sim \exp(\beta_i)$, the optimal allocation by minimization of the two risk indicators is: \[ \forall i\in \{1,\ldots,d\},~~u_i=\frac{1/\beta_i}{\sum_{j=1}^{d}1/\beta_j}u \/.\]    
\end{Example}
It is noticeable that in this particular case, the optimal allocation for comonotonic risks coincide with the asymptotic allocation in the independence case (see Proposition \ref{pro3.2}).
 \begin{Example}[Comonotonic log-normal model]
For comonotonic risks of log-normal distributions $X_i\sim LN(\mu_i,\sigma)$, the optimal allocation by minimization of the two risk indicators is: 
\[  \forall i\in \{1,\ldots,d\},~~u_i=\frac{\exp(\mu_i)}{\sum_{\ell=1}^{d}\exp(\mu_\ell)}u \/. \]    
 \end{Example}
 \begin{Example}[Comonotonic Pareto model]
 For comonotonic risks of Pareto distributions of the same shape parameter $\alpha$: $X_i\sim Pa(\alpha, \lambda_i)$, the optimal allocation by minimization of the two risk indicators is: \[ \forall i\in \{1,\ldots,d\},~~u_i=\frac{\lambda_i}{\sum_{\ell=1}^{d}\lambda_\ell}u \/. \]      
  \end{Example}
  Contrary to the exponential case, this result does not coincide with the asymptotic behavior obtained for  independent Pareto distributions (see Proposition \ref{prop3.12}).
  
\subsection{The dependence impact with some copulas models}~~\\
~~ 
In this section, we study the impact of the dependence on the optimal capital allocation using some copulas (see Nelsen \cite{nelsen} for review on copulas). The idea is to find the optimal allocation as a function of the copula parameters in each case. We focus on the indicator $I$, the same kind of calculations can be done for the indicator $J$.
\subsubsection{FGM Bivariate Model}~~\\
Let $X_1$ and $X_2$ be two risks of marginal exponential distributions $X_i\sim \exp(\beta_i)$ and FGM bivariate dependence structure with $-1\leq\theta\leq 1$ as parameter (see Nelsen \cite{nelsen},Example 3.12., section 3.2.5). We assume that $\beta_1<\beta_2/2$.\\
In this case, the copula Pearson correlation coefficient is given by $\rho_P=\frac{\theta}{4}$, and the bivariate distribution function is:
\[ F_{X_1,X_2}(x_1,x_2)=(1-e^{-\beta_1x_1})(1-e^{-\beta_2x_2})+\theta(1-e^{-\beta_1x_1})(1-e^{-\beta_2x_2})e^{-\beta_1x_1}e^{-\beta_2x_2}\/. \]
\begin{Proposition}[The optimal capital allocation for the indicator $I$ in the FGM Model]
\label{efgm}
For the indicator $I$, the optimal allocation of a capital $u$ is given by $(\beta u,(1-\beta)u)$
such that $\beta=u_1/u$ is the unique solution in $[0,1]$ of the following equation:
\begin{align}\label{didier}
(1+2\theta)(h(\beta)-h(\alpha-\alpha\beta))&+2\theta(h(2\beta)-h(2\alpha-2\alpha\beta))\\\nonumber+(1+\theta)h(\alpha+\beta-\alpha\beta)+\theta h(2\alpha+2\beta-2\alpha\beta)-&\theta h(\alpha+2\beta-\alpha\beta)-\theta h(2\alpha+\beta-2\alpha\beta)\\\nonumber=\frac{1+\theta}{\alpha-1}(h(\alpha)+\alpha h(1))+\frac{\theta}{\alpha-1}(h(2\alpha)+\alpha h(2))-&\frac{\theta}{\alpha-2}(2h(\alpha)+\alpha h(2))-\frac{\theta}{2\alpha-1}(h(2\alpha)+2\alpha h(1))\/,
\end{align}
where, $h$ is the function $h(x)=\exp(-\beta_1ux)$,and  $\alpha=\beta_2/\beta_1$.
\end{Proposition}
\begin{proof}
The proof is postponed to Appendix \ref{FGM}.
\end{proof}
\begin{remark}
In the case of $\theta =0$, we find exactly Equation (\ref{EIZO}) given by Proposition \ref{prp21} for the independent exponential distributions model. 
\end{remark}
 Equation (\ref{didier}) gives the behavior of the optimal allocation with respect to $\theta$. It may be solved numerically.\\
Figure \ref{EFGM-case1} presents an illustration of the optimal allocation variation with respect to the dependence parameter of the FGM copula. 
 \begin{figure}[H]
                \center
                \includegraphics[width=14cm]{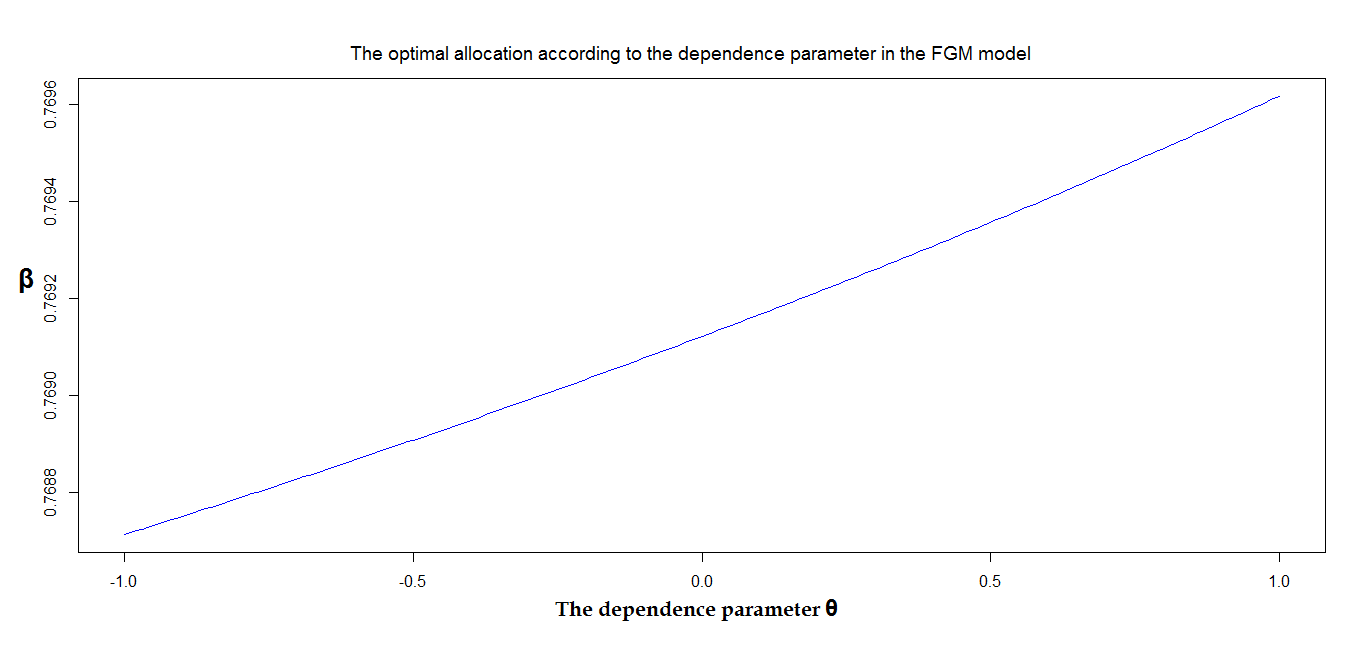} 
              	\caption{$\beta$ as a function of $\theta$. Case : $\beta_1=0.05$, $\beta_2=0.25$ , and $u=50$}
               \label{EFGM-case1}
 \end{figure}
 For the illustration parameters, we remark that $\beta$ is an increasing function of $\theta$, this can be verified analytically using the implicit function theorem. It is important to remark that the optimal allocation is also a function of the capital $u$. The variations of $\beta$ with respect to the dependence parameter depends on $u$. Figure \ref{EFGM-case2} give what we obtain as result for an allocation with the same distributions parameters but with $u=100$ as allocation capital.
  \begin{figure}[H]
                 \center
                 \includegraphics[width=14cm]{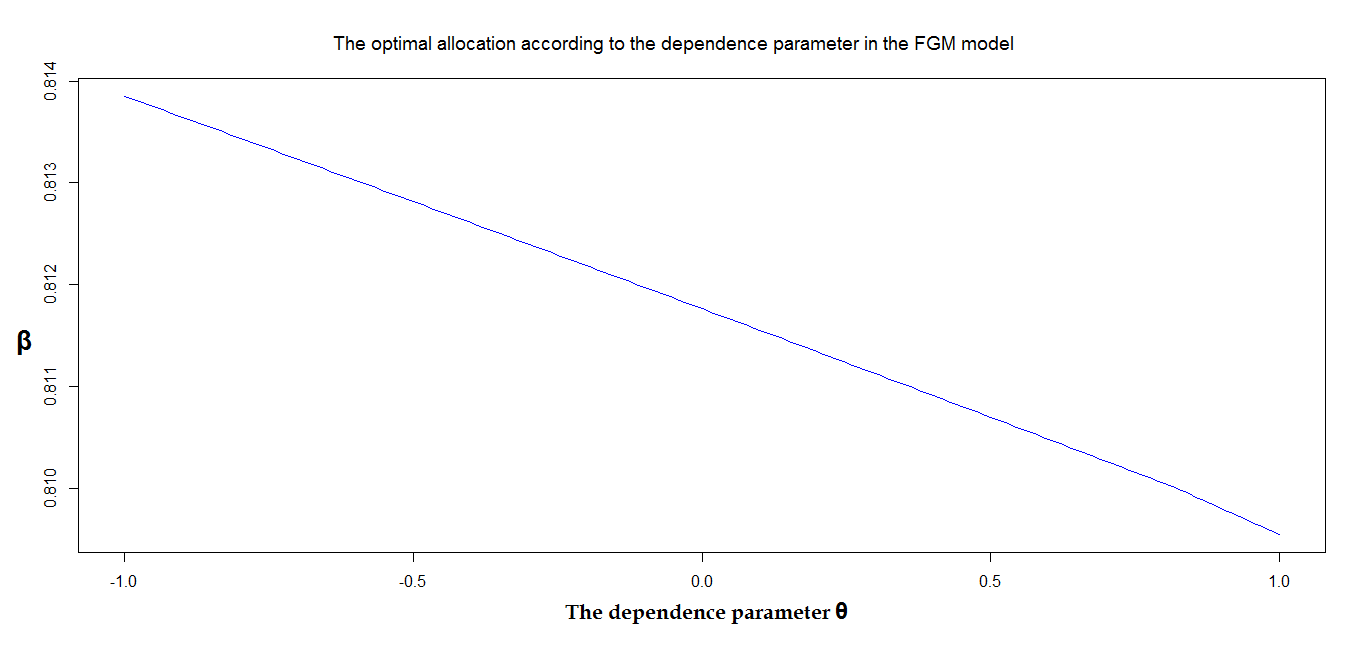} 
               	\caption{$\beta$ as a function of $\theta$. Case : $\beta_1=0.05$, $\beta_2=0.25$ , and $u=100$}
                \label{EFGM-case2}
  \end{figure}
 The optimal allocation depends on the initial capital. As an example, with the same parameters, $\beta$ is a decreasing function of $\theta$ if $u<53$, and it is an increasing function of $\theta$ for $u\geq 53$.\\ 
The variation range size of $\beta$ is very small, and its size is function of the distributions parameters and the allocation capital. That is due to the dependence structure, the FGM copula present only a light dependence. For Clayton copula as example, which is stronger dependence structure, the variation range size of $\beta$ is more important.
 \subsubsection{Marshall-Olkin Model}~~\\
 Let $Y_i\sim \exp(\lambda_i)$, with $i=0,1,2$ be three independent random variables.\\
  We construct two random variables with common shock: $X_i=\min (Y_i,Y_0)$ for $i=1,2$. $X_i$'s have exponential marginal distributions of parameters $\beta_i=\lambda_i+\lambda_0$ (see e.g. Nelsen \cite{nelsen} section 3.1.1.).\\
  This dependence construction model has as Pearson correlation coefficient
   $\rho_P=\frac{\lambda_0}{\lambda_0+\lambda_1+\lambda_2}$.\\
  The joint distribution function is given by: \begin{align*}
  \bar{F}_{X_1,X2}(x_1,x_2)&=\mathbb{P}(X_1>x_1, X_2>x_2)= \mathbb{P}(Y_1>x_1, Y_2>x_2,Y_0>\max(x_1,x_2) )\\
  &=e^{-\lambda_1x_1}e^{-\lambda_2x_2}e^{-\lambda_0\max(x_1,x_2)}\\
  &=e^{-(\lambda_0+\lambda_1)x_1}e^{-(\lambda_0+\lambda_2)x_2}e^{\lambda_0\min(x_1,x_2)}\\
  &=\bar{F}_{X_1}(x_1)\bar{F}_{X_2}(x_2)e^{\lambda_0\min(x_1,x_2)}\/,
   \end{align*} 
   and the joint density function is the following:
   \[  f_{X_1,X_2}(x_1,x_2)=\left\lbrace 
   \begin{array}{lcl} 
   f^1_{X_1,X_2}(x_1,x_2)=\beta_1e^{-\beta_1 x_1}(\beta_2-\lambda_0)e^{-(\beta_2-\lambda_0)x_2} & si & x_1>x_2\\ 
   f^2_{X_1,X_2}(x_1,x_2)=(\beta_1-\lambda_0)e^{-(\beta_1-\lambda_0)x_1}\beta_2e^{-\beta_2 x_2} & si & x_1<x_2\\ 
   f^0_{X_1,X_2}(x_1,x_2)=\lambda_0e^{-\beta_1x} e^{-\beta_2x}e^{\lambda_0x}& si & x_1=x_2=x 
   \end{array}\right.\/.\]
   \begin{Proposition}[The optimal capital allocation for the indicator $I$ in the Marshall-Olkin Model]
   We suppose that $\lambda_1<\lambda_2$. The optimal allocation of a capital $u$ minimizing the indicator $I$ is given by $(\beta u,(1-\beta)u)$,
   such that $\beta=u_1/u$ is the unique solution in $[0,1]$ of the following equation:
   \label{MOeq}
   \begin{align*}
                g(\beta_2(1-\beta))-g(\beta_1\beta)+\frac{\beta_1}{\beta_1-\lambda_2}g((\beta_1-\lambda_2)\beta+\lambda_2)+\frac{\lambda_2}{\beta_1-\lambda_2}g((\lambda_2-\beta_1)(1-\beta)+\beta_1)\\-\frac{\lambda_1}{\lambda_1-\beta_2}g(\beta_2)
                   =\frac{\lambda_2}{\beta_1-\lambda_2}g(\beta_1)+g(\lambda_s/2)[\frac{\beta_1}{\beta_1-\lambda_2}-\frac{\lambda_1}{\lambda_1-\beta_2}]
                   \/, \end{align*}  
       where, $\lambda_s=\lambda_0+\lambda_1+\lambda_2$, and $g$ is the function $g(x)= \exp(-ux)$.
  \end{Proposition}
  \begin{proof}
  The proof is postponed to Appendix \ref{MO}.
  \end{proof}
   \begin{remark}
   In the case $\lambda_0=0$ which is the independence case, we find again Equation~(\ref{EIZO}) given by Proposition \ref{prp21}.  
   \end{remark}
 We can consider $\lambda_0$ as a dependence parameter in this model. Figure \ref{MO-Corr} presents an illustration of the variation of $\beta$ as a function of $\lambda_0$.   
  \begin{figure}[H]
                       \center
                       \includegraphics[width=14cm]{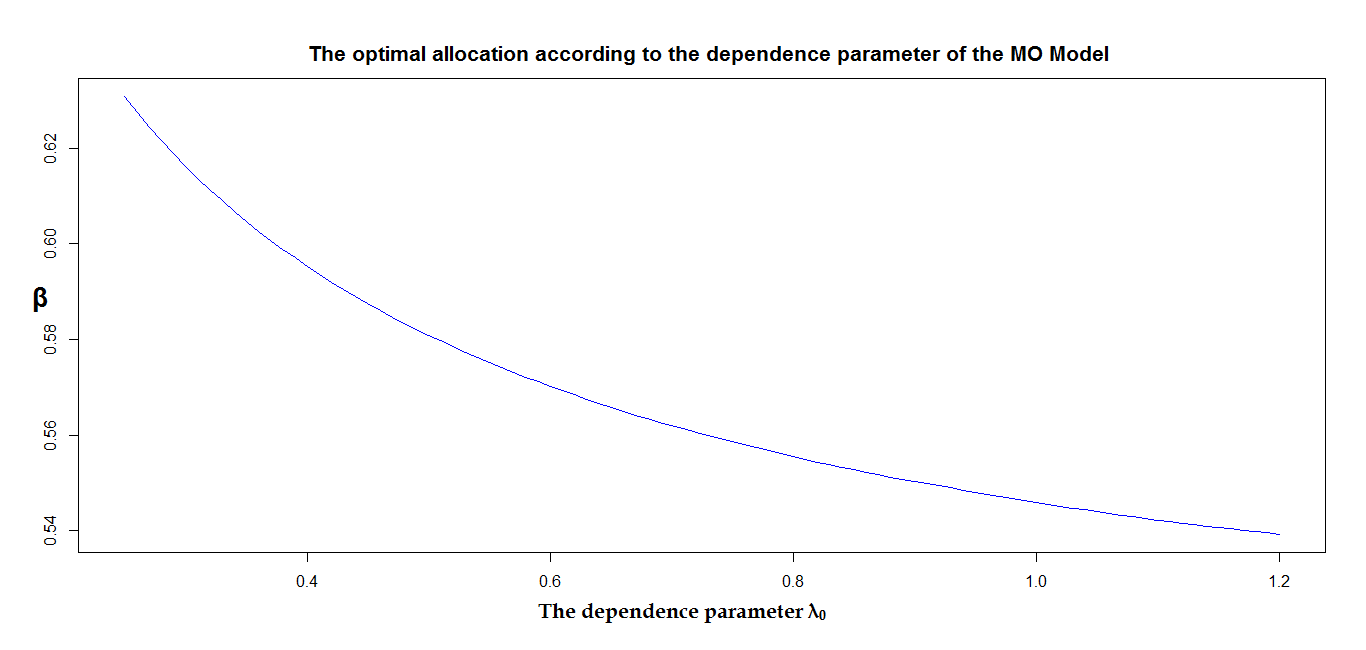} 
                     	\caption{$\beta$ as a function of $\lambda_0$ : $\beta_1=0.05$, $\beta_2=0.25$ and $u=50$. }
                      \label{MO-Corr}
        \end{figure}
  One can notice that $\beta$ is a decreasing function of $\lambda_0$. This is coherent with the increase of $\beta $ as a function of $\alpha =\frac{\beta_2}{\beta_1}=\frac{\lambda_2+\lambda_0}{\lambda_1+\lambda_0}$ demonstrated in \cite {AR1} in the independence case, since the two risks are independent, conditionally to  $Y_0 $.
 \section*{Conclusion}
  In this article, we have studied the allocation asymptotic behavior based on the level of the group capital. It has enabled us to build an idea of the capital level impact on the sensitivity of its allocation between branches. The comparison between the asymptotic optimal allocation, in the case of sub-exponential and exponential distributions, underscores the impact of the risks nature on the behavior of the allocation for a very large capital.\\
  
  Compared to the classical risk allocation methods, the allocation by minimizing multivariate risk indicator take into account the capital level. This seems to be a more acceptable behavior in a capital allocation operation.\\
  
  The risk aggregation transforms the global risk portfolio in a univariate risk. At the opposite, the allocation is based on a multivariate analysis. The goal of this paper is to highlight the importance of the dependence modeling in the success of an allocation capital operation. The capital allocation is also sensitive to the risks nature, we tried in this paper to make more evident the impact of distributions' nature on the allocation behavior. Since the capital allocation is an important financial decision, especially for groups, insurers must be very careful in their risk and dependence modeling choices to get an efficient allocation.
  
\bibliographystyle{plain}
\bibliography{pap1} 
\appendix
\section{Proofs}
\subsection{Proof of Proposition \ref{efgm}}
\label{FGM}
\begin{proof}
The bivariate density function is the following: 
\begin{align*} 
f_{X_1,X_2}(x_1,x_2)&= (1+\theta)f(x_1,x_2,\beta_1,\beta_2)+\theta f(x_1,x_2,2\beta_1,2\beta_2)\\
&-\theta f(x_1,x_2,2\beta_1,\beta_2)-\theta f(x_1,x_2,\beta_1,2\beta_2)f_{X_1,X_2}(x_1,x_2)\/,
\end{align*} 
where $f$ is the function $f(x,t,a,b)=abe^{-ax}e^{-bt}$.\\
We use the equality: $f_{X_1,S=X_i+X_2}(x_1,s)=f_{X_1,X_2}(x_1,s-x_1)1\!\!1_{\{s\geq x_1\}}$
 for all $s\geq x_1$,
to find the expression of $F_{X_1,S}(x_1,s)$, using a double integration:
\begin{align*}
F_{X_1,S}(x_1,s)&=\int_{0}^{s}\int_{0}^{x_1}f_{X_1,X_2}(x,t-x)1\!\!1_{\{t\geq x\}}dx dt
=\int_{0}^{x_1}\int_{x}^{s}f_{X_1,X_2}(x,t-x)dtdx\\
&=(1+\theta)F(x_1,s,\beta_1,\beta_2)+\theta F(x_1,s,2\beta_1,2\beta_2)
-\theta F(x_1,s,2\beta_1,\beta_2)-\theta F(x_1,s,\beta_1,2\beta_2)\/,
\end{align*}
where $F$ is the following function:
\[F(x_1,s,a,b)=\int_{0}^{x_1}\int_{x}^{s}abe^{-(a-b)x}e^{-bt}dtdx=1-e^{-ax_1}+\frac{a}{b-a}e^{-bs}-\frac{a}{b-a}e^{-bs+(b-a)x_1} \/. \]
The same way and by the symmetry of the FGM model:
\[ F_{X_2,S}(x_2,s)=(1+\theta)F(x_2,s,\beta_2,\beta_1)+\theta F(x_2,s,2\beta_2,2\beta_1)
-\theta F(x_2,s,2\beta_2,\beta_1)-\theta F(x_2,s,\beta_2,2\beta_1) \/.\]
Using $\mathbb{P}(X_i>u_i, S\leq u)= \mathbb{P}(S\leq u) - \mathbb{P}(X_i\leq u_i, S\leq u)$, the optimal allocation is the unique solution in $\mathcal{U}^2_u$ of the equation: $ F_{X_1,S}(u_1,u)=F_{X_2,S}(u_2,u)$. Then, the optimal allocation is determined by $\beta$ the solution of the equation: $F_{X_1,S}(\beta u,u)=F_{X_2,S}((1-\beta) u,u)$.\\
Since, 
\begin{align*}
F_{X_1,S}(\beta u,u)&=1+4\theta-(1+2\theta)h(\beta)-2\theta h(2\beta)\\
&+(1+\theta)\frac{1}{\alpha-1}[h(\alpha)-h(\alpha+\beta-\alpha\beta)]+\theta\frac{1}{\alpha-1}[h(2\alpha)-h(2\alpha+2\beta-2\alpha\beta)]\\
&-\theta\frac{2}{\alpha-2}[h(\alpha)-h(\alpha+2\beta-\alpha\beta)]-\theta\frac{1}{2\alpha-1}[h(2\alpha)-h(2\alpha+\beta-2\alpha\beta)]\/,
\end{align*}
and, 
\begin{align*}
F_{X_2,S}((1-\beta) u,u)&=1+4\theta-(1+2\theta)h(\alpha(1-\beta))-2\theta h(2\alpha(1-\beta))\\
&+(1+\theta)\frac{\alpha}{1-\alpha}[h(1)-h(\alpha+\beta-\alpha\beta)]+\theta\frac{\alpha}{1-\alpha}[h(2)-h(2\alpha+2\beta-2\alpha\beta)]\\
&-\theta\frac{2\alpha}{1-2\alpha}[h(1)-h(2\alpha+\beta-2\alpha\beta)]-\theta\frac{\alpha}{2-\alpha}[h(2)-h(\alpha+2\beta-\alpha\beta)]\/,
\end{align*}
we deduce from that the equation presented in the proposition \ref{efgm}.
\end{proof}
\subsection{Proof of Proposition \ref{MOeq}}
\label{MO}
\begin{proof}
The joint distribution function is given by:
   \begin{align*}
      F_{X_1,S}(x_1,s)&=\int_{0}^{s}\int_{0}^{x_1}f_{X_1,S}(x,t-x)1\!\!1_{\{t>x\}}dxdt\\
      &=\int_{0}^{s}\int_{0}^{x_1}f^1_{X_1,S}(x,t-x)1\!\!1_{\{2x>t>x\}}dxdt+\int_{0}^{s}\int_{0}^{x_1}f^2_{X_1,S}(x,t-x)1\!\!1_{\{t>2x\}}dxdt\\&+\int_{0}^{s}\int_{0}^{x_1}f^0_{X_1,S}(x,t-x)1\!\!1_{\{t=2x\}}dxdt\/.
         \end{align*}
   we distinguish between two cases:
   \paragraph*{Case $s>2x_1$:} in this case, 
    \begin{align*}
       \int_{0}^{s}\int_{0}^{x_1}f^1_{X_1,S}(x,t-x)1\!\!1_{\{2x>t>x\}}dxdt&= \int_{2x_1}^{s}\int_{0}^{x_1}f^1_{X_1,S}(x,t-x)1\!\!1_{\{2x>t>x\}}dxdt\\&+\int_{0}^{2x_1}\int_{0}^{x_1}f^1_{X_1,S}(x,t-x)1\!\!1_{\{2x>t>x\}}dxdt\\
       &=\int_{0}^{2x_1}\int_{0}^{x_1}f^1_{X_1,S}(x,t-x)1\!\!1_{\{2x>t>x\}}dxdt\\
       &=\int_{0}^{2x_1}\int_{t/2}^{min(x_1,t)}f^1_{X_1,S}(x,t-x)dxdt\\
       &=\int_{0}^{x_1}\int_{t/2}^{t}f^1_{X_1,S}(x,t-x)dxdt+\int_{x_1}^{2x_1}\int_{t/2}^{x_1}f^1_{X_1,S}(x,t-x)dxdt\/,
          \end{align*}
   and,  \begin{align*}
          \int_{0}^{s}\int_{0}^{x_1}f^2_{X_1,S}(x,t-x)1\!\!1_{\{t>2x\}}dxdt&=\int_{0}^{x_1}\int_{0}^{s}f^2_{X_1,S}(x,t-x)1\!\!1_{\{t>2x\}}dtdx=\int_{0}^{x_1}\int_{2x}^{s}f^2_{X_1,S}(x,t-x)dtdx\/,
       \end{align*}
   and, \begin{align*}
             \int_{0}^{s}\int_{0}^{x_1}f^0_{X_1,S}(x,t-x)1\!\!1_{\{t=2x\}}dxdt&= \frac{\lambda_0}{\lambda_0+\lambda_1+\lambda_2}(1-e^{-(\lambda_0+\lambda_1+\lambda_2)x_1})\/,
          \end{align*}
   then, we deduce the explicit expression of $F_{X_1,S}(x_1,s)$:
   \begin{align*}
     F_{X_1,S}(x_1,s)&=\frac{2\beta_1\lambda_2}{(\beta_1-\lambda_2)(\beta_1+\lambda_2)}(1-e^{-(\beta_1+\lambda_2)x_1})-\frac{\lambda_2}{\beta_1-\lambda_2}(1-e^{-\beta_1x_1})-\frac{\beta_1}{\beta_1-\lambda_2}(e^{-\beta_1x_1}-e^{-(\beta_1+\lambda_2)x_1})\\
     &+\frac{\lambda_1}{\lambda_1+\beta_2}(1-e^{-(\lambda_1+\beta_2)x_1})-\frac{\lambda_1}{\lambda_1-\beta_2}e^{-\beta_2s}+\frac{\lambda_1}{\lambda_1-\beta_2}e^{-(\lambda_1-\beta_2)x_1-\beta_2s}\\&+\frac{\lambda_0}{\lambda_0+\lambda_1+\lambda_2}(1-e^{-(\lambda_0+\lambda_1+\lambda_2)x_1})\/.
   \end{align*} 
    \paragraph*{Case $2x_1>s>x_1$ :}
    \begin{align*}
        \int_{0}^{s}\int_{0}^{x_1}f^1_{X_1,S}(x,t-x)1\!\!1_{\{2x>t>x\}}dxdt&=\int_{0}^{s}\int_{t/2}^{min(x_1,t)}f^1_{X_1,S}(x,t-x)dxdt\\
        &=\int_{0}^{x_1}\int_{t/2}^{t}f^1_{X_1,S}(x,t-x)dxdt+\int_{x_1}^{s}\int_{t/2}^{x_1}f^1_{X_1,S}(x,t-x)dxdt\/,
           \end{align*} 
        and,
        \begin{align*}
           \int_{0}^{s}\int_{0}^{x_1}f^2_{X_1,S}(x,t-x)1\!\!1_{\{t>2x\}}dxdt&=\int_{0}^{s}\int_{0}^{t/2}f^2_{X_1,S}(x,t-x)dxdt\/,
        \end{align*}
        and,
        \begin{align*}
                     \int_{0}^{s}\int_{0}^{x_1}f^0_{X_1,S}(x,t-x)1\!\!1_{\{t=2x\}}dxdt&= \frac{\lambda_0}{\lambda_0+\lambda_1+\lambda_2}(1-e^{-(\lambda_0+\lambda_1+\lambda_2)s/2})\/,
                  \end{align*}
           then, we deduce also in this case, the explicit expression of $F_{X_1,S}(x_1,s)$: 
       \begin{align*}
              F_{X_1,S}(x_1,s)&=\frac{2\beta_1\lambda_2}{(\beta_1-\lambda_2)(\beta_1+\lambda_2)}(1-e^{-(\beta_1+\lambda_2)s/2})-\frac{\lambda_2}{\beta_1-\lambda_2}(1-e^{-\beta_1x_1})\\&-\frac{\beta_1}{\beta_1-\lambda_2}(e^{-\beta_1x_1}-e^{-(\beta_1-\lambda_2)x_1-\lambda_2s})
              +\frac{\lambda_1}{\lambda_1-\beta_2}(1-e^{-\beta_2s})\\&-\frac{2\lambda_1\beta_2}{(\lambda_1-\beta_2)(\lambda_1+\beta_2)}(1-e^{-(\lambda_1+\beta_2)s/2})+\frac{\lambda_0}{\lambda_0+\lambda_1+\lambda_2}(1-e^{-(\lambda_0+\lambda_1+\lambda_2)s/2})\/.
            \end{align*}
         We remark that $\lambda_0+\lambda_1+\lambda_2=\lambda_1+\beta_2=\lambda_2+\beta_1$, and we suppose that $\lambda_1>\lambda_2$. Using the monotony property, we deduce that $1>\beta>1/2$, then $2\beta u>u>\beta u$. So, for $u_1=\beta u$, and $g(x)=\exp(-xu)$:
         \[ F_{X_1,S}(\beta u,u) = 1-g(\beta_1\beta)-\frac{\lambda_1}{\lambda_1-\beta_2}g(\beta_2)+\frac{\beta_1}{\beta_1-\lambda_2}g((\beta_1-\lambda_2)\beta+\lambda_2)+g(\lambda_s/2)[\frac{\lambda_1}{\lambda_1-\beta_2}-\frac{\beta_1}{\beta_1-\lambda_2}]\/, \]
         and, 
         \[ F_{X_2,S}((1-\beta)u,u) = 1-g(\beta_2(1-\beta)+\frac{\lambda_2}{\lambda_2-\beta_1}[g((\lambda_2-\beta_1)(1-\beta)+\beta_1)-g(\beta_1)] \/.\]
          That is sufficient to get \ref{MOeq}.         
   \end{proof}
   \end{document}